\documentclass[reqno,a4paper, 11pt]{amsart}

\usepackage{graphicx}

\usepackage[ansinew]{inputenc}
\usepackage{amsfonts,epsfig}
\usepackage{latexsym}
\usepackage{mathabx}
\usepackage{amsmath}
\usepackage{amssymb}

\usepackage{color}

\newtheorem{theorem}{Theorem}
\newtheorem{lemma}[theorem]{Lemma}

\newtheorem{proposition}[theorem]{Proposition}

\newtheorem{lettertheorem}{Theorem}
\newtheorem{letterlemma}[lettertheorem]{Lemma}

\theoremstyle{definition}

\theoremstyle{remark}

\numberwithin{equation}{section}

\newcommand{\D}{\mathbb{D}}
\newcommand{\DD}{\widehat{\mathcal{D}}}
\newcommand{\Dd}{\widecheck{\mathcal{D}}}

\newcommand{\DDD}{\mathcal{D}}

\newcommand{\N}{\mathbb{N}}

\newcommand{\C}{\mathbb{C}}

\renewcommand{\phi}{\varphi}

\newcommand{\whw}{\widehat{\omega}}
\newcommand{\whW}{\widehat{W}}
\newcommand{\veps}{\varepsilon}

     \def\om{\omega}      
              
                  \def\z{\zeta}

\def\LL{\mathcal{L}}

\newenvironment{Prf}{\noindent{\emph{Proof of}}}
{\hfill$\Box$ }

\begin{document}
\title{Embedding theorems for Bergman-Zygmund spaces induced by doubling weights}
\allowdisplaybreaks


\thanks{The author is supported in part by Finnish Cultural Foundation, North Karelia Regional fund and Academy of Finland 356029.}

\author[A. Pennanen]{Atte Pennanen}
\address{University of Eastern Finland, P.O.Box 111, 80101 Joensuu, Finland}
\email{atte.pennanen@uef.fi}

\begin{abstract}
    Let $0<p<\infty$ and $\Psi: [0,1) \to (0,\infty)$, and let $\mu$ be a finite positive Borel measure on the unit disk $\D$ of the complex plane. We define the Lebesgue-Zygmund space $L^p_{\mu,\Psi}$ as the space of all measurable functions $f$ on $\D$ such that $\int_{\D}|f(z)|^p\Psi(|f(z)|)\,d\mu(z)<\infty$. The weighted Bergman-Zygmund space $A^p_{\om,\Psi}$ induced by a weight function $\om$ consists of analytic functions in $L^p_{\mu,\Psi}$ with $d\mu=\om\,dA$. 

    Let $0<q<p<\infty$ and let $\om$ be radial weight on $\D$ which has certain two-sided doubling properties. In this study, we will characterize the measures $\mu$ such that the identity mapping $I: A^p_{\om,\Psi} \to L^q_{\mu,\Phi}$ is bounded and compact, when we assume $\Psi,\Phi$ to be almost monotonic and to satisfy certain doubling properties. In addition, we apply our result to characterize the measures for which the differentiation operator $D^{(n)}: A^p_{\om,\Psi} \to L^q_{\mu,\Phi}$ is bounded and compact.
\end{abstract}

\maketitle

\section{Introduction}

Let $0<p<\infty$, $\Psi: [0,\infty) \to (0,\infty)$ and $\mu$ be a finite positive Borel measure on the unit disk $\D=\{z \in \C: |z|<1\}$ of the complex plane. The Lebesgue-Zygmund space $L^p_{\mu,\Psi}$ consists of measurable functions $f: \D \to \C$ such that 
\begin{equation*}
    \|f\|_{L^p_{\mu,\Psi}}^p=\int_{\D}|f(z)|^p\Psi(|f(z)|)\,d\mu(z)<\infty.
\end{equation*}
We are interested in the case when the inducing function $\Psi$ is almost monotonic (also known as essentially monotonic in some references) and satisfies 
\begin{equation*}
    \Psi(x) \asymp \Psi(x^2), \quad 0\leq x<\infty.
\end{equation*}
We denote the class of functions satisfying these properties as $\LL$. Here we say that a function $\Psi$ is almost increasing if there exists a constant $C>0$ such that $\Psi(x)\leq C\Psi(y)$ whenever $x<y$. Almost decreasing functions are defined in a comparable way. The prototypes for this class are logarithmic functions of the type $x \mapsto \log(e+x)^{\alpha}$ for $\alpha \in \mathbb{R}$. In this situation we are talking about objects related to the classical Lebesgue-Zygmund space, see for example \cite{BR} and \cite{BR2}. For basic properties of functions in $\LL$, see \cite{BergZygD}. Note that $\|\cdot\|_{L^p_{\mu,\Psi}}$ is in general not a norm even for $p>1$. However, it is a quasinormed topological vector space, see \cite{BergZygD}. Here we do not require absolute homogeneity from a quasinorm.

Let $\om: \D \to [0,\infty)$ be an integrable function over $\D$ with respect to the Lebesgue measure. We call these functions weights. The special case of the Lebesgue-Zygmund spaces we are focusing on is when the measure $\mu$ is given by $d\mu=\om\,dA$ for some weight $\om$. Here $dA$ is the normalized Lebesgue area measure. In this case, we denote $L^p_{\mu,\Psi}=L^p_{\om,\Psi}$. We denote $\mathcal{H}(\D)$ to be the space of analytic functions in $\D$, and say that $f$ belongs to the weighted Bergman-Zygmund space $A^p_{\om,\Psi}$ if $f \in L^p_{\om,\Psi} \cap \mathcal{H}(\D)$. When the inducing weight is the standard weight, meaning that $\om(z)=(\alpha+1)(1-|z|^2)^{\alpha}$ for $\alpha>-1$, and $\Psi$ is a logarithmic function, these spaces were studied in \cite{BergZyg}.

In this study we are interested in the weighted Bergman-Zygmund spaces induced by radial doubling weights. A weight $\om$ is radial if $\om(z)=\om(|z|)$ for every $z \in \D$. To define doubling weights, we introduce some more notation. We let $\whw(z)=\int_{|z|}^{1}\om(r)\,dr$ to be the tail integral defined for $z \in \D$. We proceed to assume that $\whw(z)>0$ for every $z \in \D$, as otherwise $A^p_{\om,\Psi}=\mathcal{H}(\D)$, which is certainly not interesting. A weight $\om$ belongs to the class $\DD$, if
there exists a constant $C=C(\om)>1$ such that 
\begin{equation*}
    \whw(r)\leq C\whw\left(\frac{1+r}{2}\right), \quad 0\leq r<1,
\end{equation*}
holds. In a similar manner, a weight $\om$ belongs to $\Dd$ if there exist constants $C=C(\om)>1$ and $K=K(\om)>1$ such that
\begin{equation*}
    \whw(r) \geq C\whw\left(1-\frac{1-r}{K}\right), \quad 0\leq r<1.
\end{equation*}
We denote $\DDD = \DD \cap \Dd$, and call weights in this intersection as doubling. These classes of weights are quite general and vast, as the weights are not required to be differentiable, continuous nor strictly positive. These classes of weights emerge in a natural way in the operator theory of the weighted Bergman spaces. For instance, doubling weights are closely connected to questions related to the boundedness of the Bergman projection and the Littlewood-Paley estimates \cite{PelRat2020}. More information on these classes of weights and their importance can be found, for example, in \cite{Duan} and \cite{PRV}.

An operator $T$ is said to be bounded if it maps bounded sets into bounded sets, and compact if it maps a bounded neighborhood of the origin onto a relatively compact set. We say that $\mu$ is a $(q,\Phi)$-Carleson measure for $A^p_{\om,\Psi}$ if the identity operator $I: A^p_{\om,\Psi} \to L^q_{\mu,\Phi}$ is bounded. Carleson measures for various spaces of analytic functions have been studied since the work of Carleson~\cite{Carleson}. Since then, the problem has been solved in several different spaces of analytic functions, including the weighted Bergman spaces, see for example \cite{Hastings}, \cite{Luecking1}, \cite{Luecking2}, \cite{Wu}. In \cite{BergZygD}, the authors characterized the $(q,\Phi)$-Carleson measures for $A^p_{\om,\Psi}$ in terms of simple geometric conditions when $0<p\leq q<\infty$. In this study, we solve the remaining case $0<q<p<\infty$ using methods similar to those shown in \cite{Lu1993}. 

To state our result, a bit more notation is in order. Let $\rho(z,w)=\left|\frac{z-w}{1-\overline{z}w}\right|$ be the pseudohyperbolic distance of $z,w \in \D$. We denote $\Delta(a,r)=\left\{z \in \D: \rho(a,z)<r\right\}$ to be the pseudohyperbolic disk centered at $a \in \D$ with a radius $r \in (0,1)$. It is known that $\Delta(a,r)$ coincides with the Euclidean disk $D(A,R)$, where $A=\frac{1-r^2}{1-r^2|a|^2}a$ and $R=\frac{1-|a|^2}{1-r^2|a|^2}r$. Similarly, we define the Carleson square induced by a point $a$ as 
$$S(a)=\left\{re^{it} \in \D: |a|\leq r<1,  |\text{arg}(ae^{-it})|\leq \frac{1-|a|}{2}\right\},$$
when $a \in \D \setminus \{0\}$. For convenience, we set $S(0)=\D$. For a measurable set $E$, we define $\om(E)=\int_{E}\om\,dA$. In addition, we denote $\widetilde{\om}(z)=\frac{\whw(z)}{1-|z|}$. Note that even though $\om$ may vanish in a set of positive measure, $\widetilde{\om}$ is a strictly positive function in $\D$. We also denote $\om_{[\alpha]}(z) = (1-|z|)^{\alpha}\om(z)$ for a radial weight $\om$ and $\alpha \in \mathbb{R}$. With this notation, our characterization can be stated as the following:

\begin{theorem}\label{Theorem1}
     Let $0<q<p<\infty$, $\Phi,\Psi \in \LL$ and $\om \in \DDD$, and let $\mu$ be a positive Borel measure on $\D$. Then the following statements are equivalent: 
     \begin{itemize}
         \item[(i)]$\mu$ is a $(q,\Phi)$-Carleson measure for $A^p_{\om,\Psi}$;
         \item[(ii)] $I: A^p_{\om,\Psi} \to L^q_{\mu,\Phi}$ is compact;
         \item[(iii)] The function 
         \begin{equation*}
         \Upsilon^{\mu,\om}_{\Phi,\Psi,r}(z)=
         \frac{\mu(\Delta(z,r))\Phi\left(\frac{1}{1-|z|}\right)}{\om(S(z))\Psi\left(\frac{1}{1-|z|}\right)^{\frac{q}{p}}}, \quad z \in \D,
     \end{equation*}
        belongs to $L^{\frac{p}{p-q}}_{\widetilde{\om}}$ for some (equivalently for all) $r \in (0,1)$.
     \end{itemize}
\end{theorem}
Note that as $\om \in \DDD$, we may replace $\om(S(z))$  with $\om(\Delta(z,r))$, when $r$ is sufficiently large depending on $\om$. However, this would reduce the generality of our statement as now the choice of $r$ is not restricted. Our characterization can be seen as a generalization of the one given in \cite[Theorem 2]{LRW2021} for the weighted Bergman spaces induced by doubling weights. Argument displayed there also shows that one cannot replace $L^{\frac{p}{p-q}}_{\widetilde{\om}}$ with $L^{\frac{p}{p-q}}_{\om}$ in the condition. 

Let us define the differential operator $D^{(n)}$ as $D^{(n)}(f)=f^{(n)}$. Our second main result is the characterization of when the differential operator $D^{(n)}: A^p_{\om,\Psi} \to L^q_{\mu,\Phi}$ is bounded and compact. 
\begin{theorem}\label{Theorem2}
    Let $0<q<p<\infty$, $n \in \N$, $\Phi,\Psi \in \LL$ and $\om \in \DDD$, and let $\mu$ be a positive Borel measure on $\D$. Then the following statements are equivalent:
    \begin{itemize}
        \item[(i)] $D^{(n)}: A^p_{\om,\Psi} \to L^q_{\mu,\Phi}$ is bounded;
        \item[(ii)] $D^{(n)}: A^p_{\om,\Psi} \to L^q_{\mu,\Phi}$ is compact;
        \item[(iii)] The function 
    \begin{equation*}
        \Upsilon^{n,\mu,\om}_{\Phi,\Psi,r}(z)=
         \frac{\mu(\Delta(z,r))\Phi\left(\frac{1}{1-|z|}\right)}{\om(S(z))\Psi\left(\frac{1}{1-|z|}\right)^{\frac{q}{p}}(1-|z|)^{nq}}, \quad z \in \D,
    \end{equation*}
     belongs to $L^{\frac{p}{p-q}}_{\widetilde{\om}}$ for some (equivalently for all) $r \in (0,1)$.
    \end{itemize}
\end{theorem}
We obtain this result using Theorem~\ref{Theorem1} and the well-known Littlewood-Paley estimates for weights in $\DDD$, see \cite{PelRat2020}.

The rest of the paper is organized in the following way. In Section \ref{SecAux}, we show some known results on the doubling weights and functions in the class $\LL$ for the convenience of the reader and prove some necessary auxiliary results. Finally, Section \ref{SecMain} is dedicated to proving Theorems~\ref{Theorem1} and \ref{Theorem2}.

To finish the introduction, we give some additional notation already used. We say that $A\lesssim B$ or $B \gtrsim A$, if $A \leq CB$ for some constant $C=C(\cdot)>0$. If $A \lesssim B \lesssim A$, we denote it by $A \asymp B$. Note that the omitted constant is not shared between different instances and may depend on some fixed parameters.

\section{Auxiliary results}\label{SecAux}
The proof of Theorem \ref{Theorem1} requires knowledge of weights in $\DDD$ and functions in the class $\LL$. Therefore, for the convenience of the reader, we show some needed elementary results and give references to the interested readers. Our first lemma gives helpful characterizations for weights in $\DD$ and $\Dd$. The proof of (i) can be found for example in \cite[Lemma 2.1]{PelSum14}, and for a proof of (ii) see \cite[Lemma B]{PelDela}.
\begin{letterlemma}\label{DhatDcheck}
Let $\om$ be a radial weight. Then the following statements are valid:
\begin{itemize}
    \item[(i)] $\om \in \DD$ if and only if there exist constants $C=C(\om)>0$ and $\beta_0=\beta_0(\om)>0$ such that
		$$
		\widehat{\om}(r)\leq C\left(\frac{1-r}{1-t}\right)^{\beta}\widehat{\om}(t),\quad 0\leq r\leq t<1,
		$$
for all $\beta\geq\beta_0$.
    \item[(ii)] $\om\in\Dd$ if and only if there exist constants $C=C(\omega)>0$ and $\alpha=\alpha(\omega)>0$ such that
	$$
	\widehat{\omega}(t) \leq C\left(\frac{1-t}{1-r}\right)^{\alpha} \widehat{\omega}(r), \quad 0 \leq r \leq t<1.
	$$
\end{itemize}
\end{letterlemma}
In addition to properties of weights, we constantly need the following lemmas from \cite[Lemma 3, Proposition 4]{BergZygD}, which yield some important pointwise estimates for functions in the class $\LL$. 
\begin{letterlemma}\label{Lgrowth}
    Let $\Psi \in \LL$. Then there exist constants $c_1=c_1(\Psi)>0$, $C_1=C_1(\Psi)>0$ and $c_2=c_2(\Psi) \in \mathbb{R}$, $C_2=C_2(\Psi) \in \mathbb{R}$ such that
    \begin{equation*}
        c_1\left(\log(e+x)\right)^{c_2} \leq \Psi(x) \leq C_1\left(\log(e+x)\right)^{C_2}, \quad 0\leq x <\infty.
    \end{equation*}
\end{letterlemma}

\begin{letterlemma}\label{Lproperties}
   Let $0<\alpha<\infty$ and $\Phi, \Psi\in\LL$. Then the following statements hold:
   \begin{itemize}
       \item [(i)] $\Psi(x)\asymp \Psi(y)$, when $x \asymp y$;
       \item[(ii)]  $\Psi(x) \asymp \Psi(x^{\alpha})$ for all $0\leq x<\infty$;
       \item[(iii)] $\Psi(x/\Phi(x)) \asymp \Psi(x)$ for all $0\leq x<\infty$.
   \end{itemize}
\end{letterlemma}

First, we need some auxiliary results, which show that the function $z \mapsto (1-|z|)^{-\veps}$ belongs to $L^1_{\mu}$ for a sufficiently small $\veps>0$ under certain conditions on $\mu$.
\begin{lemma}\label{lemma1}
Let $0<q<p<\infty$, $\Phi,\Psi \in \LL$ and $\om \in \DDD$, and let $\mu$ be a positive Borel measure on $\D$. Assume that there exists a constant $r \in (0,1)$ such that
\begin{equation}\label{bndcondition}
    \int_{\D}\left(\frac{\mu(\Delta(z,r))}{\om(S(z))}\right)^{\frac{p}{p-q}}\left(\frac{\Phi\left(\frac{1}{1-|z|}\right)}{\Psi\left(\frac{1}{1-|z|}\right)^{\frac{q}{p}}}\right)^{\frac{p}{p-q}}\frac{\whw(z)}{1-|z|}\,dA(z) < \infty.
\end{equation}
Then, there exists a constant $\veps=\veps(\mu,\om,\Phi,\Psi,p,q)>0$ such that
\begin{equation*}
    \int_{\D}\frac{d\mu(z)}{(1-|z|)^{\veps}}<\infty.
\end{equation*}
\end{lemma}

\begin{proof}
Let $W(z)=\om(z)\Psi\left(\frac{1}{1-|z|}\right)\Phi\left(\frac{1}{1-|z|}\right)^{-\frac{p}{q}}$. By applying proof of \cite[Proposition 11]{BergZygD} twice, we see that $W \in \DDD$. We also have $\widehat{W}(z)\asymp \whw(z)\Psi\left(\frac{1}{1-|z|}\right)\Phi\left(\frac{1}{1-|z|}\right)^{-\frac{p}{q}}$ by \cite[afternote of Lemma 10]{BergZygD}. As $W \in \DDD$, it is known that $\widetilde{W} \in \DDD$, see \cite[Proof of Proposition 5]{PRS1}. Therefore, there exists a constant $\veps_0=\veps_0(\om,\Psi,\Phi,p,q)>0$ such that $\widetilde{W}_{\left[\frac{-\veps p}{q}\right]} \in L^1$ for every $0<\veps<\veps_0$ by \cite[Lemma 2]{PR2023}, where we have applied the lemma to the weight $\widetilde{W}$ and the auxiliary weight $\nu$ defined by $\nu(z)=\left(1-|z|\right)^{\frac{p}{q}-1}$. Then Fubini's theorem and H\"older's inequality yield
\begin{equation*}
    \begin{split}
        \int_{\D}\frac{d\mu(z)}{(1-|z|)^{\veps}}&\asymp \int_{\D}\frac{\mu(\Delta(\z,r))}{(1-|\z|)^{2+\veps}}\,dA(\z)
        =\int_{\D}\frac{\mu(\Delta(\z,r))}{(1-|\z|)^{2+\veps}}\frac{\widehat{W}(\z)}{\widehat{W}(\z)}\,dA(\z)\\
        &\asymp \int_{\D}\frac{\mu(\Delta(\z,r))}{W(S(\z))}\frac{\widetilde{W}(\z)}{(1-|\z|)^{\veps}}\,dA(\z)\\
        &\leq \left(\int_{\D}\left(\frac{\mu(\Delta(\z,r))}{W(S(\z))}\right)^{\frac{p}{p-q}}\widetilde{W}(\z)\,dA(\z)\right)^{\frac{p-q}{p}}
        \left(\int_{\D}\frac{\widetilde{W}(\z)}{(1-|\z|)^{\frac{\veps p}{q}}}\,dA(\z)\right)^{\frac{q}{p}}\\
        &\lesssim \left(\int_{\D}\left(\frac{\mu(\Delta(\z,r))}{W(S(\z))}\right)^{\frac{p}{p-q}}\widetilde{W}(\z)\,dA(\z)\right)^{\frac{p-q}{p}}.
    \end{split}
\end{equation*}
By the definition of $W$ and the assumption \eqref{bndcondition}, we have
\begin{equation}
\begin{split}
&\quad \int_{\D}\left(\frac{\mu(\Delta(\z,r))}{W(S(\z))}\right)^{\frac{p}{p-q}}\widetilde{W}(\z)\,dA(\z)\\ &\asymp \int_{\D}\left(\frac{\mu(\Delta(\z,r))}{\om(S(\z))}\right)^{\frac{p}{p-q}}\left(\frac{\Phi\left(\frac{1}{1-|\z|}\right)}{\Psi\left(\frac{1}{1-|\z|}\right)^{\frac{q}{p}}}\right)^{\frac{p}{p-q}}\frac{\whw(\z)}{1-|\z|}\,dA(\z) < \infty,
\end{split}
\end{equation}
which finishes the proof.
\end{proof}

In the proof of the necessity of our condition \eqref{bndcondition}, we once again need an auxiliary lemma similar to that of Lemma~\ref{lemma1}. The result is certainly not surprising given the previous one. However, we need some specific test functions to obtain it.
\begin{lemma}\label{epsilonmulemma}
    Let $0<q<p<\infty$, $\Phi,\Psi \in \LL$ and $\om \in \DDD$, and let $\mu$ be a positive Borel measure on $\D$. If $\mu$ is a $(q,\Phi)$-Carleson measure for $A^p_{\om,\Psi}$, then there exists a constant $\veps=\veps(\mu,\om,\Phi,\Psi,p,q)>0$ such that
\begin{equation*}
    \int_{\D}\frac{d\mu(z)}{(1-|z|)^{\veps}}<\infty.
\end{equation*}
\end{lemma}
\begin{proof}
    As $\om \in \DDD$ by the hypothesis, there exists a constant $\delta_0=\delta_0(\om)>0$ such that $\om_{[-\delta]} \in L^1$ for every $0<\delta<\delta_0$ by \cite[Lemma 2]{PR2023}. Further, for every $\delta>0$ there exist functions $f_1,f_2 \in \mathcal{H}(\D)$ such that
    \begin{equation*}
        |f_1(z)|+|f_2(z)| \asymp \frac{1}{(1-|z|)^{\delta}}, \quad z \in \D,
    \end{equation*}
    by \cite[Theorem 1]{GPR}. Here, the interested reader should also see \cite{AD}, \cite{Ramey} and the references therein. Fix $\delta>0$ such that $\delta<\frac{\delta_0}{p}$. Then, by the almost monotonicity of $\Psi$ and Lemma~\ref{Lproperties}(i), we have
    \begin{equation*}
    \begin{split}
                \int_{\D}|f_j(z)|^p\Psi(|f_j(z)|)\om(z)\,dA(z) &\lesssim \int_{\D}\frac{\max\left\{1,\Psi\left(\frac{1}{1-|z|}\right)\right\}}{(1-|z|)^{p\delta}}\om(z)\,dA(z)\\
                &\lesssim\int_{\D}\frac{\om(z)}{(1-|z|)^{p\delta+\eta}}\,dA(z)<\infty, \quad j=1,2,
    \end{split}
    \end{equation*}
    where $\eta>0$ comes from the growth estimate of functions in $\LL$ by Lemma~\ref{Lgrowth} and can be chosen small enough such that $\eta<\delta_0-p\delta$. Therefore $f_1,f_2 \in A^p_{\om,\Psi}$ and we can use the assumption that $\mu$ is a $(q,\Phi)$-Carleson measure, which yields
    \begin{equation*}
        \begin{split}
            \infty &> \int_{\D}|f_j(z)|^q\Phi\left(|f_j(z)|\right)\,d\mu(z)\gtrsim\int_{\D}|f_j(z)|^q\min\left\{1,\Phi\left(\frac{1}{1-|z|}\right)\right\}\,d\mu(z), \quad j=1,2.
        \end{split}
    \end{equation*}
    Here we have used the almost monotonicity of functions in $\LL$ and Lemma~\ref{Lproperties}(i)-(ii). Therefore, we have
    \begin{equation*}
        \begin{split}
            \infty&>\int_{\D}\left(|f_1(z)|^q+|f_2(z)|^q\right)\min\left\{1,\Phi\left(\frac{1}{1-|z|}\right)\right\}\,d\mu(z)\\
            &\gtrsim \int_{\D}\frac{d\mu(z)}{(1-|z|)^{q\delta-\gamma}}.
        \end{split}
    \end{equation*}
    Here $\gamma>0$ comes from the growth estimate of $\Phi \in \LL$ in Lemma~\ref{Lgrowth} and can be chosen such that $\gamma<q\delta$. Thus, we conclude that
    \begin{equation*}
        \int_{\D}\frac{d\mu(z)}{(1-|z|)^{q\delta-\gamma}}<\infty,
    \end{equation*}
    which proves the claim.
\end{proof}
For the proofs of our main results, we need the following growth estimate for functions in $A^p_{\om,\Psi}$:
    \begin{equation}\label{growthf}
        |f(z)|^p \lesssim \frac{1}{\om(S(z))\Psi\left(\frac{1}{\om(S(z))}\right)}, \quad  z \in \D, 
    \end{equation}
    see \cite[Lemma 12]{BergZygD} for the proof. Note that the omitted constant depends on the quasinorm $\|f\|_{A^p_{\om,\Psi}}$. We also note that there exists a constant $\beta=\beta(\om)>0$ such that $(1-|z|)^{1+\beta}\lesssim\om(S(z)) \lesssim (1-|z|)$ for every $z \in \D$ by Lemma~\ref{DhatDcheck}. 
\section{Main results}\label{SecMain}
Using the previous results, we first show that \eqref{bndcondition} is a sufficient condition for the compactness of $I: A^p_{\om,\Psi} \to L^q_{\mu,\Phi}$. 
\begin{proposition}\label{3imp1}
    Let $0<q<p<\infty$, $\Phi,\Psi \in \LL$ and $\om \in \DDD$, and let $\mu$ be a positive Borel measure on $\D$. Then $I: A^p_{\om,\Psi} \to L^q_{\mu,\Phi}$ is compact if there exists a constant $r \in (0,1)$ such that 
    \begin{equation}\label{assump}
        \int_{\D}\left(\frac{\mu(\Delta(z,r))}{\om(S(z))}\right)^{\frac{p}{p-q}}\left(\frac{\Phi\left(\frac{1}{1-|z|}\right)}{\Psi\left(\frac{1}{1-|z|}\right)^{\frac{q}{p}}}\right)^{\frac{p}{p-q}}\frac{\whw(z)}{1-|z|}\,dA(z) < \infty.
    \end{equation}  
\end{proposition}
\begin{proof}
Let $\{f_j\}$ be a norm bounded sequence in $A^p_{\om,\Psi}$ such that $f_j\to 0$ uniformly on compact subsets of $\D$. It is well-known that to prove the proposition, we need to show that $\|f_j\|^q_{L^q_{\mu,\Phi}} \to 0$ as $j \to \infty$, see for instance \cite[Proposition 8]{BergZygD} for a proof in the case of the weighted Bergman-Zygmund spaces. Let $\veps>0$. For every $R \in (0,1)$ we have 
\begin{equation*}
    \begin{split}
        \|f_j\|_{L^q_{\mu,\Phi}}^q&=\int_{\D}|f_j(z)|^q\Phi\left(|f_j(z)|\right)\,d\mu(z)\\
                                  &=\int_{D(0,R)}|f_j(z)|^q\Phi\left(|f_j(z)|\right)\,d\mu(z)+\int_{\D \setminus D(0,R)}|f_j(z)|^q\Phi\left(|f_j(z)|\right)\,d\mu(z).
    \end{split}
\end{equation*}
As $\mu$ is a finite measure by Lemma \ref{lemma1}, the dominated convergence theorem shows that there exists an $N \in \N$ such that 
$\int_{D(0,R)}|f_j|^q\Phi\left(|f_j|\right)\,d\mu<\veps$ for all $j \geq N(\veps,R)$. Therefore, we only need to consider the second integral. Assume first that $\Phi$ is almost decreasing. Fix $\delta>0$ such that $\int_{\D}\frac{d\mu(z)}{(1-|z|)^{\delta}}<\infty$, which can be done by Lemma~\ref{lemma1}. Thus, there exists $\rho=\rho(\veps) \in (0,1)$ such that $\int_{\D \setminus D(0,\rho)}\frac{d\mu(z)}{(1-|z|)^{\delta}}<\veps$. Define the set $E_{\delta}(f_j)=\left\{z \in \D: |f_j(z)|<\frac{1}{(1-|z|)^{\delta/q}}\right\}$. Therefore, using Lemma~\ref{Lproperties}(ii) we deduce
\begin{equation}\label{splitdisk}
\begin{split}
    \|f_j\|^q_{L^q_{\mu,\Phi}}&\leq \veps +\int_{\D \setminus D(0,R)}|f_j(z)|^q\Phi(|f_j(z)|)\,d\mu(z)\\
    &\lesssim \veps + \int_{\left(\D \setminus D(0,R) \right) \cap E_{\delta}}\frac{d\mu(z)}{(1-|z|)^{\delta}}\\
    &\quad  + \int_{\left(\D \setminus D(0,R)\right) \setminus E_{\delta}}|f_j(z)|^q\Phi\left(\frac{1}{1-|z|}\right)\,d\mu(z).
    \end{split}
\end{equation}
In the case of the integral over $(\D \setminus D(0,R)) \cap E_{\delta}$, we have
\begin{equation*}
    \int_{\left(\D \setminus D(0,R) \right) \cap E_{\delta}}\frac{d\mu(z)}{(1-|z|)^{\delta}}<\veps,
\end{equation*}
when we fix $R>\rho$. For the integral over $\left(\D \setminus D(0,R)\right) \setminus E_{\delta}$, we use the subharmonicity of $|f_j|^q$, Fubini's theorem, Lemma~\ref{Lproperties}(i) and H\"older's inequality, which for $R>r$ yield
\begin{equation*}
\begin{split}
    &\quad \int_{\left(\D \setminus D(0,R)\right) \setminus E_{\delta}}|f_j(z)|^q\Phi\left(\frac{1}{1-|z|}\right)\,d\mu(z)\\
    &\lesssim \int_{\D \setminus D(0,R)}\left(\int_{\Delta(z,r)}\frac{|f_j(\zeta)|^q}{(1-|\zeta|)^2}\Phi\left(\frac{1}{1-|\zeta|}\right)\frac{\whw(\zeta)}{\whw(\zeta)}\,dA(\zeta)\right)\,d\mu(z)\\
    &\asymp \int_{\D}|f_j(\zeta)|^q\Phi\left(\frac{1}{1-|\zeta|}\right)\frac{\Psi\left(\frac{1}{1-|\zeta|}\right)^{\frac{q}{p}}}{\Psi\left(\frac{1}{1-|\zeta|}\right)^{\frac{q}{p}}}\frac{\whw(\zeta)}{1-|\zeta|}\frac{\mu\left(\Delta(\zeta,r)\cap \left(\D \setminus D(0,R)\right)\right)}{\om(S(\zeta))}\,dA(\zeta)\\
    &\leq \int_{\D \setminus D\left(0,\frac{R-r}{1-Rr}\right)}|f_j(\zeta)|^q\Phi\left(\frac{1}{1-|\zeta|}\right)\frac{\Psi\left(\frac{1}{1-|\zeta|}\right)^{\frac{q}{p}}}{\Psi\left(\frac{1}{1-|\zeta|}\right)^{\frac{q}{p}}}\frac{\whw(\zeta)}{1-|\zeta|}\frac{\mu\left(\Delta(\zeta,r)\right)}{\om(S(\zeta))}\,dA(\zeta)\\
    &\leq \left(\int_{\D \setminus D\left(0,\frac{R-r}{1-Rr}\right)}|f_j(\zeta)|^p\Psi\left(\frac{1}{1-|\zeta|}\right)\frac{\whw(\zeta)}{1-|\zeta|}\,dA(\zeta)\right)^{\frac{q}{p}}\\
    &\quad \cdot \left(\int_{\D \setminus D\left(0,\frac{R-r}{1-Rr}\right)}\left(\frac{\mu(\Delta(\zeta,r))}{\om(S(\z))}\right)^{\frac{p}{p-q}}\left(\frac{\Phi\left(\frac{1}{1-|\z|}\right)}{\Psi\left(\frac{1}{1-|\z|}\right)^{\frac{q}{p}}}\right)^{\frac{p}{p-q}}\frac{\whw(\z)}{1-|\z|}\,dA(\z)\right)^{\frac{p-q}{p}}\\
    &\lesssim \left(\|f_j\|_{A^p_{\om,\Psi}}^q+1\right)\\
    &\quad \cdot\left(\int_{\D \setminus D\left(0,\frac{R-r}{1-Rr}\right)}\left(\frac{\mu(\Delta(\z,r))}{\om(S(\z))}\right)^{\frac{p}{p-q}}\left(\frac{\Phi\left(\frac{1}{1-|\z|}\right)}{\Psi\left(\frac{1}{1-|\z|}\right)^{\frac{q}{p}}}\right)^{\frac{p}{p-q}}\frac{\whw(\z)}{1-|\z|}\,dA(\z)\right)^{\frac{p-q}{p}}.
\end{split}
\end{equation*}
Here the first inequality follows by removing the set of integration where $\Delta(\zeta,r)\cap \left(\D \setminus D(0,R)\right)=\emptyset$ and using the estimate $\mu\left(\Delta(\zeta,r)\cap \left(\D \setminus D(0,R)\right)\right) \leq \mu\left(\Delta(\zeta,r)\right)$. The last asymptotic inequality follows by first estimating the area of integration to the whole disk, using integration by parts in the same way as was done in \cite[Proof of Theorem 1(i), (3.8)]{BergZygD} and then either using the estimate \eqref{growthf} or the method shown in \eqref{splitdisk} together with Lemma~\ref{Lproperties}(i)-(ii) depending on the monotonicity of $\Psi$.
For every fixed $r \in (0,1)$, the function 
$g(R)=\frac{R-r}{1-Rr}$ is increasing in $(0,1)$ and $g(R) \to 1$, as $R \to 1^-$. By the assumption \eqref{assump}, we deduce that
\begin{equation}\label{Rcondition}
    \left(\int_{\D \setminus D\left(0,\frac{R-r}{1-Rr}\right)}\left(\frac{\mu(\Delta(\z,r))}{\om(S(\z))}\right)^{\frac{p}{p-q}}\left(\frac{\Phi\left(\frac{1}{1-|\z|}\right)}{\Psi\left(\frac{1}{1-|\z|}\right)^{\frac{q}{p}}}\right)^{\frac{p}{p-q}}\frac{\whw(\z)}{1-|\z|}\,dA(\z)\right)^{\frac{p-q}{p}}<\veps,
\end{equation}
whenever $\frac{R-r}{1-Rr}>R_0$, where $R_0=R_0(\veps) \in (0,1)$. Therefore, \eqref{Rcondition} holds for every $R>\frac{R_0+r}{1+R_0r}$. Thus if $R>\max\left\{r, \rho, \frac{R_0+r}{1+R_0r}\right\}$, we have
\begin{equation*}
    \|f_j\|_{L^q_{\mu,\Phi}} \lesssim \veps,
\end{equation*}
for every $j \geq N(R,\veps)=N(\veps)$, which concludes the proof in the case of $\Phi$ being almost decreasing. In the case of almost increasing $\Phi$, the proof follows the same reasoning and the only difference is that we can ignore the step of splitting the disk in \eqref{splitdisk} and instead use the growth estimate \eqref{growthf} and Lemma~\ref{Lproperties}(i)-(iii). This concludes the proof.       
\end{proof}

With these preparations we are ready to prove Theorem~\ref{Theorem1}.

\medskip

\Prf \emph{Theorem~\ref{Theorem1}}.
   Since every compact operator is bounded, (ii) implies (i). Further, by Proposition~\ref{3imp1}, (iii) implies (ii), so it remains to prove that (i) implies (iii). Define $W(z)=\Psi\left(\frac{1}{1-|z|}\right)\om(z)$. Under the hypothesis $\om \in \DDD$, it is known that $W \in \DDD$ and $\|f\|_{A^p_{\om,\Psi}}\lesssim 1+\|f\|_{A^p_{W}}$ in such a way that the spaces contain the same functions by \cite[Proposition 11]{BergZygD}. Then \cite[Theorem 1]{PRS3} gives us an atomic decomposition for the Bergman spaces induced by the weight~$W$. More precisely, for $\{a_{k}\} \in \ell^p$ and a sufficiently large constant $M=M(p,\om,\Psi)>0$ one has
    \begin{equation*}
        f(z)=\sum_{k=0}^{\infty}a_{k}\frac{(1-|z_{k}|^2)^{M-\frac{1}{p}}\whW(z_{k})^{-\frac{1}{p}}}{(1-\overline{z_{k}}z)^M}, \quad z \in \D,
    \end{equation*}
    in such a way that $\|\{a_{k}\}\|_{\ell^p} \gtrsim \|f\|_{A^p_W}$. Here $\{z_{k}\}$ is a sequence separated in the pseudohyperbolic metric, meaning $\inf_{j\neq k}\rho(z_j,z_k)>0$. Assume first that $\Phi$ is almost decreasing. Then we consider the $L^q_{\mu,\Phi}$-quasinorm of the function $f$. By the growth of $|f|$ from \eqref{growthf} and Lemma~\ref{Lproperties}(i)-(iii), we have
    \begin{equation}\label{eq1}
        \begin{split}
            C&\geq \int_{\D}|f(z)|^q\Phi(|f(z)|)\,d\mu(z) \gtrsim \int_{\D}|f(z)|^q\Phi\left(\frac{1}{1-|z|}\right)\,d\mu(z)\\
            &= \int_{\D}\left|\sum_{k=0}^{\infty}a_{k}\frac{(1-|z_{k}|^2)^{M-\frac{1}{p}}\whW(z_{k})^{-\frac{1}{p}}}{(1-\overline{z_{k}}z)^M}\right|^q\Phi\left(\frac{1}{1-|z|}\right)\,d\mu(z),
        \end{split}
    \end{equation}
    where $C=C(p,q,\om,\mu,\Psi,\Phi,\|\{a_{k}\}\|_{\ell^p})>0$ comes from the assumption that $\mu$ is a $(q,\Phi)$-Carleson measure for $A^p_{\om,\Psi}$. Assume next that $\Phi$ is almost increasing. Define the set $E_{\veps}(f)=\left\{z \in \D: |f(z)|<(1-|z|)^{-\frac{\veps}{2q}}\right\}$, where $\veps>0$ is the same as in Lemma~\ref{epsilonmulemma}. In this case, we have
    \begin{equation}\label{normestimate+1}
        \begin{split}
            \int_{\D}|f(z)|^q\Phi\left(\frac{1}{1-|z|}\right)\,d\mu(z)&\lesssim\int_{E_{\veps}(f)}\frac{1}{(1-|z|)^{\frac{\veps}{2}}}\Phi\left(\frac{1}{1-|z|}\right)\,d\mu(z)\\
            &\quad +\int_{\D \setminus E_{\veps}(f)}|f(z)|^q\Phi\left(|f(z)|\right)\,d\mu(z)\\
            &\lesssim\int_{E_{\veps}(f)}\frac{1}{(1-|z|)^{\frac{\veps}{2}+\eta}}\,d\mu(z)\\
            &\quad +\int_{\D \setminus E_{\veps}(f)}|f(z)|^q\Phi\left(|f(z)|\right)\,d\mu(z)\\
            &\lesssim 1+\int_{\D}|f(z)|^q\Phi\left(|f(z)|\right)\,d\mu(z),
        \end{split}
    \end{equation}
    where $\eta>0$ comes from Lemma~\ref{Lgrowth} and is chosen small enough so that $\eta<\frac{\veps}{2}$. Here we have also used Lemma~\ref{Lproperties}. These estimates yield us the same conclusion as was shown in \eqref{eq1} with a different constant $C$. Next, we replace $a_k$ with $a_kr_k(t)$, where $r_k(t)$ is the $k$th Rademacher function. Then integrating both sides from $0$ to $1$ with respect to $t$ and applying Khinchine's inequality together with Fubini's theorem yields
     \begin{equation*}
         \begin{split}
             C &\gtrsim \int_{\D}\left(\sum_{k=0}^{\infty}|a_{k}|^2\frac{(1-|z_{k}|^2)^{2M-\frac{2}{p}}\whW(z_{k})^{-\frac{2}{p}}}{|1-\overline{z_{k}}z|^{2M}}\right)^{\frac{q}{2}}\Phi\left(\frac{1}{1-|z|}\right)\,d\mu(z) \\
             &\gtrsim  \int_{\D}\left(\sum_{k=0}^{\infty}|a_{k}|^2\chi_{\Delta(z_{k},R)}(z)(1-|z_{k}|^2)^{-\frac{2}{p}}\whW(z_{k})^{-\frac{2}{p}}\right)^{\frac{q}{2}}\Phi\left(\frac{1}{1-|z|}\right)\,d\mu(z) \\
             &\asymp \sum_{k=0}^{\infty}|a_{k}|^q\mu(\Delta(z_{k},R))(1-|z_{k}|^2)^{-\frac{q}{p}}\whW(z_k)^{-\frac{q}{p}}\Phi\left(\frac{1}{1-|z_k|}\right)\\
             &\asymp \sum_{k=0}^{\infty}|a_{k}|^q\frac{\mu(\Delta(z_{k},R))\Phi\left(\frac{1}{1-|z_{k}|}\right)}{\left(\Psi\left(\frac{1}{1-|z_{k}|}\right)\om(S(z_k))\right)^{\frac{q}{p}}},
         \end{split}
     \end{equation*}
    for any fixed $R>0$, which the omitted constant depends on. Here we have also used Lemma~\ref{Lproperties}. We also need to note that each $z \in \D$ can only be contained in finitely many pseudohyperbolic disks $\Delta(z_{k},R)$ when $\{z_{k}\}$ is separated and $\inf_{j\neq k}\rho(z_j,z_k)<R$, see for instance~\cite[Lemma 3]{Lu1993}. Let $\{b_k\}=\{|a_k|^q\}$, which therefore belongs to $\ell^{\frac{p}{q}}$.
    Next we want to prove that 
    \begin{equation*}
     \sum_{k=0}^{\infty}\left(\frac{\mu(\Delta(z_{k},R))\Phi\left(\frac{1}{1-|z_{k}|}\right)}{\left(\Psi\left(\frac{1}{1-|z_{k}|}\right)\om(S(z_k))\right)^{\frac{q}{p}}}\right)^{\frac{p}{p-q}}<\infty.
     \end{equation*}
     This can be proved using duality. Assume that $\|\{b_k\}\|_{\ell^{\frac{p}{q}}}= 1$. Then, we define a linear functional $T: \ell^{\frac{p}{q}} \to \mathbb{R}$ as
     \begin{equation*}
         T(\{b_k\})=\sum_{k=0}^{\infty}b_{k}\frac{\mu(\Delta(z_{k},R))\Phi\left(\frac{1}{1-|z_{k}|}\right)}{\left(\Psi\left(\frac{1}{1-|z_{k}|}\right)\om(S(z_k))\right)^{\frac{q}{p}}}.
     \end{equation*}
     By the previous remarks, this operator is clearly bounded as the upper bound $C$ is uniform, and by duality one has 
     \begin{equation*}
         \frac{\mu(\Delta(z_{k},R))\Phi\left(\frac{1}{1-|z_{k}|}\right)}{\left(\Psi\left(\frac{1}{1-|z_{k}|}\right)\om(S(z_k))\right)^{\frac{q}{p}}} \in \ell^{\frac{p}{p-q}}.
     \end{equation*}
     From here, we want to prove the continuous version. Note that this argument holds for every $R \in (0,1)$ and for every sequence $\{z_k\}$ that is separated such that $\inf_{j\neq k}\rho(z_j,z_k)<R$. Let $r \in (0,1)$ be fixed and let $R=\frac{1+r}{2}$. Consider a separated sequence $\{z_k\}$ such that $\inf_{j \neq k}\rho(z_j,z_k)=\frac{1-r}{2}>0$ such that $\bigcup_k \Delta\left(z_k,\frac{1-r}{2}\right)=\D$. Sequences like this exist, see for example \cite[Lemma 4]{Lu1993}. Therefore, Lemma~\ref{DhatDcheck} yields
     \begin{equation*}
         \begin{split}
             &\quad \int_{\D}\left(\frac{\mu(\Delta(z,r))}{\om(S(z))}\right)^{\frac{p}{p-q}}\left(\frac{\Phi\left(\frac{1}{1-|z|}\right)}{\Psi\left(\frac{1}{1-|z|}\right)^{\frac{q}{p}}}\right)^{\frac{p}{p-q}}\frac{\whw(z)}{1-|z|}\,dA(z)\\
             &\leq \sum_{k=0}^{\infty}\int_{\Delta\left(z_k,\frac{1-r}{2}\right)}\left(\frac{\mu(\Delta(z,r))}{\om(S(z))}\right)^{\frac{p}{p-q}}\left(\frac{\Phi\left(\frac{1}{1-|z|}\right)}{\Psi\left(\frac{1}{1-|z|}\right)^{\frac{q}{p}}}\right)^{\frac{p}{p-q}}\frac{\whw(z)}{1-|z|}\,dA(z)\\
             &\leq \sum_{k=0}^{\infty}\mu(\Delta(z_k,R))^{\frac{p}{p-q}}\int_{\Delta\left(z_k,\frac{1-r}{2}\right)}\frac{1}{\om(S(z))^{\frac{p}{p-q}}}\left(\frac{\Phi\left(\frac{1}{1-|z|}\right)}{\Psi\left(\frac{1}{1-|z|}\right)^{\frac{q}{p}}}\right)^{\frac{p}{p-q}}\frac{\whw(z)}{1-|z|}\,dA(z)\\
             &\asymp \sum_{k=0}^{\infty}\left(\frac{\mu(\Delta(z_k,R))}{\om(S(z_k))}\right)^{\frac{p}{p-q}}\left(\frac{\Phi\left(\frac{1}{1-|z_k|}\right)}{\Psi\left(\frac{1}{1-|z_k|}\right)^{\frac{q}{p}}}\right)^{\frac{p}{p-q}}\int_{\Delta\left(z_k,\frac{1-r}{2}\right)}\frac{\whw(z)}{1-|z|}\,dA(z)\\
             &\lesssim \sum_{k=0}^{\infty}\left(\frac{\mu(\Delta(z_k,R))}{\om(S(z_k))}\right)^{\frac{p}{p-q}}\left(\frac{\Phi\left(\frac{1}{1-|z_k|}\right)}{\Psi\left(\frac{1}{1-|z_k|}\right)^{\frac{q}{p}}}\right)^{\frac{p}{p-q}}\widetilde{\om}(S(z_k))\\
             &\asymp \sum_{k=0}^{\infty}\left(\frac{\mu(\Delta(z_k,R))}{\om(S(z_k))}\right)^{\frac{p}{p-q}}\left(\frac{\Phi\left(\frac{1}{1-|z_k|}\right)}{\Psi\left(\frac{1}{1-|z_k|}\right)^{\frac{q}{p}}}\right)^{\frac{p}{p-q}}\om(S(z_k))\\
             &=\sum_{k=0}^{\infty}\left(\frac{\mu(\Delta(z_{k},R))\Phi\left(\frac{1}{1-|z_{k}|}\right)}{\left(\Psi\left(\frac{1}{1-|z_{k}|}\right)\om(S(z_k))\right)^{\frac{q}{p}}}\right)^{\frac{p}{p-q}}<\infty,
         \end{split}
     \end{equation*}
     which concludes the proof.
\hfill$\Box$ \newline \newline
\medskip
Next, using Theorem~\ref{Theorem1} we will prove Theorem~\ref{Theorem2}.
\medskip \newline
\begin{Prf}\emph{Theorem~\ref{Theorem2}}.
    As compact operators are clearly bounded, (ii) implies (i). Assume first (i) and we prove that it implies (iii). As $\om \in \DDD$ by the hypothesis, the weight $\om_{[\alpha]}(z) = (1-|z|)^{\alpha}\om(z)$ belongs to $\DDD$ for $\alpha>0$ by the definition of class $\DDD$ and Lemma~\ref{DhatDcheck}, see also \cite[Lemma 3]{PR2024} for a more general result. Therefore, by \cite[Proof of Proposition 11]{BergZygD} the weight $z \mapsto \Psi\left(\frac{1}{1-|z|}\right)(1-|z|)^{\alpha}\om(z)$ belongs to $\DDD$ for every $\alpha>0$. Let $f \in A^p_{\om_{[np]},\Psi}$ such that $\|f\|_{A^p_{\om_{[np]},\Psi}}^p\leq M$ for some $M>0$. It is well-known that there exists a function $g \in \mathcal{H}(\D)$ such that $g^{(n)}=f$ and $g^{(j)}(0)=0$ for $j \in \{0,1,\dots,n-1\}$. Now $g \in A^p_{\om,\Psi}$ and there exists a constant $K=K(M,p,\om,\Psi,n)>0$ such that $\left\|g\right\|_{A^p_{\om,\Psi}} \leq K$. This fact follows directly from the boundedness of the identity mappings between $A^p_{\om,\Psi}$ and its corresponding weighted Bergman space $A^p_{W}$, where $W(z)=\Psi\left(\frac{1}{1-|z|}\right)\om(z)$, see \cite[Proposition 11]{BergZygD}, and the Littlewood-Paley formula for weights in $\DDD$, see \cite[Theorem 5]{PelRat2020}. Thus, by the assumption that $D^{(n)}: A^p_{\om,\Psi} \to L^q_{\mu,\Phi}$ is bounded, it is clear that 
    \begin{equation*}
    \begin{split}
                 \int_{\D}|f(z)|^q\Phi(|f(z)|)\,d\mu(z)&=\int_{\D}\left|g^{(n)}(z)\right|^q\Phi\left(\left|g^{(n)}(z)\right|\right)\,d\mu(z)\\
                 &\leq C(K)=C(M,p,q,\om,\mu,\Psi,\Phi,n).
    \end{split}
    \end{equation*}
    Therefore, $I: A^p_{\om_{[np]},\Psi} \to L^q_{\mu,\Phi}$ is bounded. Hence, Theorem~\ref{Theorem1} yields
    \begin{equation*}
    \int_{\D}\left(\frac{\mu(\Delta(z,r))\Phi\left(\frac{1}{1-|z|}\right)}{\om_{[np]}(S(z))\Psi\left(\frac{1}{1-|z|}\right)^{\frac{q}{p}}}\right)^{\frac{p}{p-q}}\widetilde{\om_{[np]}}(z)\,dA(z)<\infty,
    \end{equation*}
    for every fixed $r \in (0,1)$. Next, we note that $\widehat{\om_{[np]}}(z) \asymp \whw(z)(1-|z|)^{np}$ holds for every $z \in \D$ by the definition of class $\DDD$, see once again \cite[Lemma 3]{PR2024} for a more general result. This estimate gives
    \begin{equation*}
    \begin{split}
        \left(\frac{\mu(\Delta(z,r))\Phi\left(\frac{1}{1-|z|}\right)}{\om_{[np]}(S(z))\Psi\left(\frac{1}{1-|z|}\right)^{\frac{q}{p}}}\right)^{\frac{p}{p-q}}\widetilde{\om_{[np]}}(z) 
        &\asymp \left(\frac{\mu(\Delta(z,r))\Phi\left(\frac{1}{1-|z|}\right)}{\om(S(z))\Psi\left(\frac{1}{1-|z|}\right)^{\frac{q}{p}}(1-|z|)^{np}}\right)^{\frac{p}{p-q}}\widetilde{\om}(z)(1-|z|)^{np}\\
        &\asymp \left(\frac{\mu(\Delta(z,r))\Phi\left(\frac{1}{1-|z|}\right)}{\om(S(z))\Psi\left(\frac{1}{1-|z|}\right)^{\frac{q}{p}}(1-|z|)^{nq}}\right)^{\frac{p}{p-q}}\widetilde{\om}(z), \quad z \in \D,
    \end{split}
    \end{equation*}
    which proves (iii).
    
    Next, we assume (iii). It is well-known that to prove the compactness, it is sufficient to show that for each bounded sequence $\{f_j\}$ in $ A^p_{\om,\Psi}$, which converges to $0$ uniformly in compact subsets of $\D$, we have
    $\lim_{j \to \infty}\left\|f_j^{(n)}\right\|_{L^q_{\mu,\Phi}}=0$. For the convenience of the reader, we show a proof of this claim following \cite[Proposition 8]{BergZygD}. Let $\sup_{j \in \N}\|f_j\|_{A^p_{\om,\Psi}}\leq M$ for some $M>0$.
Clearly $\{f_j\}$ is uniformly bounded on compact subsets of $\D$ by the estimate \eqref{growthf}. Thus, a normal family argument proves that there exists a subsequence $\{f_{j_k}\}$ converging to an analytic function $f$ uniformly in compact subsets of $\D$. As $\|\cdot\|_{A^p_{\om,\Psi}}$ is a quasinorm, standard arguments yield that $f \in A^p_{\om,\Psi}$. By Weierstrass's theorem, $f_{j_k}^{(n)} \to f^{(n)}$ uniformly in compact subsets of $\D$. Thus, 
$
\lim_{k \to \infty}\left\|D^{(n)}(f_{j_k})-f^{(n)}\right\|_{L^q_{\mu,\Phi}}=0
$
by the assumption, which proves that $D^{(n)}: A^p_{\om,\Psi} \to L^q_{\mu,\Psi}$ is compact. 
    
We are now ready to prove the final implication. Let $\{f_j\}$ be a sequence in $A^p_{\om,\Psi}$ which satisfies $\sup_{j \in \N}\|f_j\|_{A^p_{\om,\Psi}}\leq M$ for some $M>0$, such that $f_j$ tends to $0$ uniformly in compact subsets of $\D$ and let $\veps>0$. This implies that $f_j^{(n)}$ tends to $0$ uniformly in compact subsets of $\D$. Therefore, we use exactly the same reasoning as was used in the proof of Proposition~\ref{3imp1}, which yields
\begin{equation*}
\begin{split}
        \left\|f_j^{(n)}\right\|_{L^q_{\mu,\Phi}}^q&=\int_{\D}\left|f_j^{(n)}(z)\right|^q\Phi\left(\left|f_j^{(n)}(z)\right|\right)\,d\mu(z)\lesssim \veps 
        + \left(\left\|f_j^{(n)}\right\|_{A^p_{\om_{[np]},\Psi}}^q+1\right)\\
        &\quad \cdot \left(\int_{\D \setminus D\left(0,\frac{R-r}{1-Rr}\right)}\left(\frac{\mu(\Delta(\z,r))}{\om_{[np]}(S(\z))}\right)^{\frac{p}{p-q}}\left(\frac{\Phi\left(\frac{1}{1-|\z|}\right)}{\Psi\left(\frac{1}{1-|\z|}\right)^{\frac{q}{p}}}\right)^{\frac{p}{p-q}}\frac{\widehat{\om_{[np]}}(\z)}{1-|\z|}\,dA(\z)\right)^{\frac{p-q}{p}},
\end{split}
\end{equation*}
where $r$ is chosen such that $\Upsilon^{n,\mu,\om}_{\Phi,\Psi,r} \in L^{\frac{p}{p-q}}_{\widetilde{\om}}$, and $0<R<1$ and $j$ are large enough, see Theorem~\ref{Theorem2}(iii) and Proof of Proposition~\ref{3imp1}, respectively. Then \cite[Proposition 11]{BergZygD} and the Littlewood-Paley formula for weights in $\DDD$ together with the assumption $\sup_{j \in \N}\|f_j\|_{A^p_{\om,\Psi}} \leq M$  shows that $\sup_{j \in \N}\left\|f_j^{(n)}\right\|_{A^p_{\om_{[np]},\Psi}}^q\leq C(M,p,q,\om,\Psi,n)$. We conclude the proof by following exactly the same reasoning as was done on the proof of Proposition~\ref{3imp1}.
\end{Prf}

\section*{Acknowledgement}
The author wishes to express his gratitude to professor Jouni R\"atty\"a for his valuable comments on the manuscript. In addition, the author is grateful for the many helpful comments given by the anonymous referee, which significantly improved the exposition of the manuscript.

\end{document}